\documentclass[11pt,twoside,a4paper]{amsart}

\usepackage[utf8]{inputenc}
\usepackage{amssymb, amsthm, amsmath}
\usepackage{amsxtra}

\DeclareMathOperator{\lin}{span}
\DeclareMathOperator{\Dom}{Dom}

\newtheorem{fakt}{Proposition}
\newtheorem{tw}{Theorem}
\newtheorem{wn}{Corollary}
\newtheorem{hipoteza}{Conjecture}

\newtheorem{thmx}{Theorem}

\newtheorem{hipotezaX}{Conjecture}

\newtheorem{wnX}{Corollary}

\theoremstyle{remark}
\newtheorem{uwaga}{Remark}

\theoremstyle{remark}
\newtheorem{przyklad}{Example}

\author{Grzegorz Świderski}
\address{
	Grzegorz Świderski\\
	Instytut Matematyczny\\
	Uniwersytet Wrocławski\\
	Pl. Grunwaldzki 2/4\\
	50-384 Wrocław\\
	Poland}
\email{gswider@math.uni.wroc.pl}

\title[Spectral properties of Jacobi matrices]{Spectral properties\\ of unbounded Jacobi matrices with \\ almost monotonic weights}

\subjclass[2010]{Primary: 47B36, 42C05. Secondary: 60J80}

\keywords{Jacobi matrix, continuous spectrum, Chihara's conjecture}

\begin{document}
   \date{}
   \maketitle

   \begin{abstract}
		We present an unified framework to identify spectra of Jacobi matrices. We give applications to long-standing conjecture of Chihara (\cite{Chihara1}, \cite{Chihara2}) concerning one-quarter class of orthogonal polynomials, to the conjecture posed by Roehner and Valent \cite{BRGV} concerning continuous spectra of generators of birth and death processes and to spectral properties of operators studied by Janas, Moszyński \cite{JM1} and Pedersen \cite{SP}.
   \end{abstract}
   
	\section{Introduction}
      Given sequences $\{a_n\}_{n=0}^\infty$ and $\{b_n\}_{n=0}^\infty$ such that $a_n > 0$ and $b_n \in \mathbb{R}$ we set
      \begin{equation*}
         C =
            \left( 
               \begin{array}{cccccc}
               b_0 & a_0 & 0   & 0   & 0   &\ldots \\
               a_0 & b_1 & a_1 & 0   & 0   & \ldots \\
               0   & a_1 & b_2 & a_2 & 0   & \ldots \\
               0   & 0   & a_2 & b_3 & a_3 & \ldots \\
               \vdots & \vdots & \vdots & \vdots & \vdots& \ddots
               \end{array} 
            \right).
      \end{equation*}
      The operator~$C$ is defined on the domain $\Dom(C) = \{ x \in \ell^2 \colon C x \in \ell^2\}$, where
      \[
         \ell^2 = \{ x \in \mathbb{C}^{\mathbb{N}} \colon \sum_{n=0}^\infty |x_n|^2 < \infty \}
      \]
      and is called a \emph{Jacobi matrix}.
   
      The study of Jacobi matrices is motivated by connections with orthogonal polynomials and classical moment problem (see e.g. \cite{BS1}). Also every self-adjoint operator can be represented as a direct sum of Jacobi matrices. In particular, generators of birth and death processes may be seen as Jacobi matrices acting on weighted $\ell^2$ spaces. 
      
      There are several approaches to the problem of the indentification of the spectrum of unbounded Jacobi matrices. A method often used is based on subordination theory (see e.g. \cite{SLC}, \cite{JN1}, \cite{MM}). Another technique uses the analysis of commutator between Jacobi matrix and a suitable chosen matrix (see e.g. \cite{JS1}). The case of Jacobi matrices with monotonic weights was considered mainly by Dombrowski (see e.g. \cite{JD1}), where the author developed commutator techniques which enabled qualitative spectral analysis of examined operators.
		
		The present article is motivated by commutator techniques of Dombrowski and some ideas of Clark~\cite{SLC}. In fact, commutators do not appear here directly but are hidden in some of our expressions.

      Let $C$ be a Jacobi matrix and assume that the matrix~$C$ is self-adjoint. The spectrum of the operator~$C$ will be denoted by $\sigma(C)$, the set of all its eigenvalues by $\sigma_p(C)$ and the set of all accumulation points of $\sigma(C)$ by $\sigma_{ess}(C)$. For a real number $x$ we define $x^- = \max(-x,0)$.
		
		Our main result is the following theorem.
		\begin{thmx} \label{twSpektrumOgolne}
         Let $C$ be a Jacobi matrix. If there is a positive sequence $\{ \alpha_n \}$ such that
         \begin{flalign}
            \tag{a} &\lim_{n \rightarrow \infty} a_n = \infty, &\\
            \tag{b} &\sum_{n=1}^\infty \left[ \frac{a_{n+1}}{a_n} \frac{\alpha_{n+1}}{\alpha_n} - \frac{a_n}{a_{n-1}} \frac{\alpha_{n-1}}{\alpha_n} \right]^- < \infty, \\
            \tag{c} &\sum_{n=1}^\infty \frac{1}{a_{n-1}} \left| \frac{a_{n-1}}{a_n} - \frac{\alpha_{n-1}}{\alpha_n} \right| < \infty, \\
            \tag{d} &\sum_{n=0}^\infty \left| \frac{b_{n+1}}{a_n} - \frac{b_n}{a_{n-1}} \frac{\alpha_{n-1}}{\alpha_n} \right| < \infty, \\
            \tag{e} &\sum_{n=0}^\infty \frac{1}{a_n \alpha_n} = \infty, \\
            \tag{f} &\lim_{n \rightarrow \infty} \frac{\alpha_{n-1}}{\alpha_n} \frac{a_n}{a_{n-1}} = 1, \\
            \tag{g} &\limsup_{n \rightarrow \infty} \frac{|b_n|}{a_n} < 2
         \end{flalign}
			then the Jacobi matrix~$C$ is self-adjoint and satisfies $\sigma_p(C) = \emptyset$, and $\sigma(C) = \mathbb{R}$.
		\end{thmx}
		
		The importance of Theorem~\ref{twSpektrumOgolne} lies in the fact that we have a flexibility in the choice of the sequence $\alpha_n$. Some choices of the sequence $\alpha_n$ are given in Section~\ref{sec:specialCases}. The simplest case is the following result.
      \begin{wnX} \label{twB}
         Assume
         \begin{flalign}
            \tag{a} &\lim_{n \rightarrow \infty} a_n = \infty, &\\
            \tag{b} &\sum_{n=0}^\infty \frac{1}{a_n^2} = \infty, \\
            \tag{c} &\sum_{n=0}^\infty \left[ \left( \frac{a_{n+1}}{a_n} \right)^2 - 1 \right]^- < \infty, \\
            \tag{d} &\limsup_{n \rightarrow \infty} \frac{|b_n|}{a_n} < 2, \\
            \tag{e} &\sum_{n=0}^\infty \frac{|b_{n+1} - b_n|}{a_n} < \infty.
         \end{flalign}
         Then the Jacobi matrix $C$ is self-adjoint and satisfies $\sigma_p(C) = \emptyset$ and $\sigma(C) = \mathbb{R}$.
      \end{wnX}
      
      In \cite[Lemma 2.6]{JDSP1} it was proven that if the nonnegative sequence $a_n^2 - a_{n-1}^2$ is bounded and $b_n \equiv 0$ then the matrix~$C$ has no eigenvalues. Corollary~\ref{twB} gives additional information that in this case holds $\sigma(C) = \mathbb{R}$. Moreover, the assumptions of Corollary~\ref{twB} are weaker than the conditions of \cite[Lemma 2.6]{JDSP1}. 
      
      In Section~\ref{sec:Examples} we provide examples showing sharpness of Corollary~\ref{twB}. In particular, condition~(b) is necessary in the class of monotonic sequences~$\{ a_n \}$ and condition~(c) could not be replaced by $[(a_{n+1}/a_n)^2 - 1]^- \rightarrow 0$. Corollary~\ref{twC} shows that in general condition~(d) is necessary. Unfortunately, we do not know whether condition~(d) is implied by the rest of the assumptions. Author knows only examples satisfying assumptions of Corollary~\ref{twB} when $|b_n|/a_n \rightarrow 0$.
      
      In Section~\ref{sec:TheoremAAppl} we apply Corollary~\ref{twB} to resolve a conjecture (see \cite{BRGV}) about continuous spectra of generators of birth and death processes. We also present there applications to the following conjecture.
      
      \begin{hipotezaX}[Chihara, \cite{Chihara1}, \cite{Chihara2}] \label{hipChiharaX}
         Assume that a Jacobi matrix~$C$ is self-adjoint, $b_n \rightarrow \infty$, the smallest point $\rho$ of $\sigma_{ess}(C)$ is finite and
         \[
            \lim_{n \rightarrow \infty} \frac{a_n^2}{b_n b_{n+1}} = \frac{1}{4}.
         \]
         Then $\sigma_{ess}(C) = [\rho, \infty)$.
      \end{hipotezaX}
      
      A direct consequence of Corollary~\ref{twB} providing easy to check additional assumptions to Conjecture~\ref{hipChiharaX} is the following result.
      
		\begin{wnX} \label{twC}
         Assume
         \begin{flalign}
            \tag{a} &\lim_{n \rightarrow \infty} a_n = \infty, &\\
            \tag{b} &\sum_{n=0}^\infty \frac{1}{a_n} = \infty, \\
            \tag{c} &\sum_{n=0}^\infty \left[ \frac{a_{n+1}}{a_n} - 1 \right]^- < \infty, \\
            \tag{d} &\lim_{n \rightarrow \infty} [a_{n-1} - b_n + a_n] = M.
         \end{flalign}
         Then the Jacobi matrix $C$ satisfies $\sigma_{ess}(C) = [-M, \infty)$. Moreover, if $a_{n+1}/a_n \rightarrow 1$ then
         \[
            \lim_{n \rightarrow \infty} \frac{a_n^2}{b_n b_{n+1}} = \frac{1}{4}.
         \]
		\end{wnX}
      
      Let us present ideas behind the proof of Theorem~\ref{twSpektrumOgolne}. Let the difference operator $J$ be defined by 
      \[
         (J x)_n = -i \alpha_{n-1} x_{n-1} + i \alpha_n x_{n+1} \quad (n \geq 0)
      \]
      for a positive sequence $\{ \alpha_n \}_{n=0}^\infty$ and $\alpha_{-1}=x_{-1}=0$. Then we define commutator $K$ on finite sequences by the formula
      \[
         -2iK = C J - J C.
      \]
      The expression $S_n = \langle K(p^n), p^n \rangle$, where $p^n = (p_0, p_1, \ldots, p_n, 0, 0, \ldots)$, $\{p_k\}$ is the formal eigenvector of $C$ and $\langle \cdot, \cdot \rangle$ is the scalar product on $\ell^2$, proved to be an useful tool to show that the matrix~$C$ has continuous spectrum (see e.g. \cite{JD1}, \cite{JDSP1}, \cite{JDSP3}).
      
      Important observation is that we can give closed form for $S_n$ (see \eqref{SN1}). To the author's knowledge this closed form has been known only for $\alpha_n = a_n$ (see \cite{JD3}). Related expression for $\alpha_n \equiv 1$ was analysed in \cite{SLC}. Adaptation of techniques from \cite{SLC} allow us to circumvent technical difficulties present in Dombrowski's approach. Extending definition of $S_n$ to generalized eigenvectors (see \eqref{defUWektorWlasny}) enable us to show that $\sigma(C) = \mathbb{R}$.
      
		The article is organized as follows: in Section~\ref{sec:Tools} we present definitions and well-known facts important for our argument. In Section~\ref{sec:TheoremA} we prove Theorem~\ref{twSpektrumOgolne}, whereas in Section~\ref{sec:specialCases} we show its variants. In particular, we identify spectra of operators considered in \cite{SP} and  \cite{JM1}. In Section~\ref{sec:TheoremAAppl} we present applications of Corollary~\ref{twB} to some open problems. Finally, in the last section we discuss the necessity of the assumptions of Corollary~\ref{twB}. We present also examples showing that in some cases Corollary~\ref{twB} is stronger than results known in the literature.
      
      \subsection*{Acknowledgments}
         The author would like to thank Ryszard Szwarc and Bartosz Trojan for their helpful suggestions concerning the presentation of this article.
	
   \section{Tools} \label{sec:Tools}
      Given a Jacobi matrix~$C$, $\lambda \in \mathbb{R}$ and real numbers $(a,b) \neq (0,0)$ we introduce a generalized eigenvector $\{ u_n \}$ by asking
		\begin{equation} \label{defUWektorWlasny}
         \begin{gathered}
            u_0 = a, \quad u_1 = b, \\
            a_n u_{n+1} = (\lambda - b_n) u_n - a_{n-1} u_{n-1} \quad (n \geq 1).
         \end{gathered}
		\end{equation}
		Furthermore we define the sequence of polynomials
      \begin{equation} \label{defWielomianyOrtogonalne}
         \begin{gathered}
            p_{-1}(\lambda) = 0, \quad p_0(\lambda) = 1, \\
            a_n p_{n+1}(\lambda) = (\lambda - b_n) p_n(\lambda) - a_{n-1} p_{n-1}(\lambda) \quad (n \geq 0).
         \end{gathered}
      \end{equation}
      The sequence $\{ p_n(\lambda) \}$ is a formal eigenvector of matrix $C$ associated with an eigenvalue $\lambda$. 
      
      Observe that $\{ p_n(\cdot) \}_{n=0}^\infty$ is a sequence of polynomials. Moreover, the sequence is orthonormal with respect to the measure $\mu(\cdot) = \langle E(\cdot) \delta_0, \delta_0 \rangle$, where $E$ is the spectral resolution of the matrix~$C$, $\langle \cdot, \cdot \rangle $ is the scalar product on $\ell^2$ and $\delta_0 = (1, 0, 0, \ldots)$.
      
		The following propositions are well-known. We include them for the sake of completeness.
		
		\begin{fakt} \label{spektrumUogolnioneWektoryWlasne}
			Let $\lambda \in \mathbb{R}$. If every generalized eigenvector $\{ u_n \}$ does not belong to $\ell^2$ then the matrix $C$ is self-adjoint, $\lambda \notin \sigma_p(C)$ and $\lambda \in \sigma(C)$.
		\end{fakt}
		\begin{proof}
			\cite[Theorem~3]{BS1} asserts that $C$ is self-adjoint provided that at least one generalized eigenvector $\{u_n\} \notin \ell^2$. Direct computation shows that $\lambda \in \sigma_p(C)$ if and only if $\{ p_n(\lambda) \} \in \ell^2$. Therefore the matrix $C$ is self-adjoint and $\lambda \notin \sigma_p(C)$.
		
			Observe that the vector $x$ such that $(C - \lambda I) x = \delta_0$ satisfies the following recurrence relation
			\[
            \begin{gathered}
               b_0 x_0 + a_0 x_1 = \lambda x_0 + 1, \\
               a_{n-1} x_{n-1} + b_n x_n + a_n x_{n+1} = \lambda x_n \quad (n \ge 1).
            \end{gathered}
         \]
			Hence $x$ is a generalized eigenvector, thus $x \notin \ell^2$. Therefore the operator $C - \lambda I$ is not surjective, i.e. $\lambda \in \sigma(C)$.
		\end{proof}

      \begin{fakt} \label{odbicieSpektrum}
         Let $C$ and $\widehat{C}$ be Jacobi matrices defined by sequences $\{a_n\}$, $\{b_n\}$ and $\{a_n\}$, $\{-b_n\}$ respectively. Then
         \[
            \sigma(C) = -\sigma(\widehat{C}), \quad \sigma_p(C) = -\sigma_p(\widehat{C}),
         \]
      \end{fakt}
      \begin{proof}
         Let $U$ be the diagonal matrix with a sequence $\{ (-1)^n \}_{n=0}^\infty$ on the main diagonal. From the identity 
         \[
            U C U^{-1} = - \widehat{C}
         \]
         and equality of domains the conclusion follows.
      \end{proof}
      
      \begin{fakt} \label{restrykcjeC}
         Let $C$ be a self-adjoint Jacobi matrix associated with the sequence $b_n \equiv 0$. Let $C_e$ and $C_o$ be restrictions of $C \cdot C$ to the subspaces $\lin\{\delta_{2k} \colon k \in \mathbb{N}\}$ and $\lin\{\delta_{2k+1} \colon k \in \mathbb{N}\}$ respectively. Then $C_e$ and $C_o$ are Jacobi matrices associated with
         \begin{equation} \label{wzoryNaC2}
            \begin{gathered}
               a_n^e = a_{2n} a_{2n+1}, \quad b_n^e = a_{2n-1}^2 + a_{2n}^2\\
               a_n^o = a_{2n+1} a_{2n+2}, \quad b_n^o = a_{2n}^2 + a_{2n+1}^2.
            \end{gathered}
         \end{equation}
         respectively. Moreover, $C_o$ and $C_e$ are self-adjoint and 
         \[
            \sigma(C_o) = \sigma(C_e) = \left( \sigma(C) \right)^2, \quad \sigma_p(C_o) = \sigma_p(C_e) = \left( \sigma_p(C) \right)^2,
         \]
         when $0 \notin \sigma_p(C)$ and $0 \notin \sigma_p(\widetilde{C})$, where $\widetilde{C}$ is a self-adjoint Jacobi matrix associated with the sequences $\{a_{n+1}\}_{n=0}^\infty$ and $\widetilde{b}_n \equiv 0$, and for a set $X$ we define $X^2 = \{ x^2 \colon x \in X \}$.
      \end{fakt}
      \begin{proof}
         By direct computation it may be proved that $C_o$ and $C_e$ satisfies \eqref{wzoryNaC2}.

         Let $\{p_n^e\}$ be the sequence of associated polynomials to the matrix $C_e$. Then \cite[Theorem~3]{BS1} asserts that $C_e$ is self-adjoint provided $\{p_n^e(0)\} \notin \ell^2$. It is known that $p_{2n}(x) = p_n^e(x^2)$ (see e.g. \cite[Section~4]{JDSP2}). Since $p_{2k+1}(0)=0$ and $0 \notin \sigma_p(C)$ we have
         \[
            \infty = \sum_{n=0}^\infty p_n^2(0) = \sum_{n = 0}^\infty p^2_{2n}(0) = \sum_{n=0}^\infty \left( p_n^e(0) \right)^2.
         \]
         Therefore $C_e$ is self-adjoint.
         
         Assume that $0 \notin \sigma_p(\widetilde{C})$. Observe that $C_o = \widetilde{C}_e$. Therefore the previous argument applied to $\widetilde{C}$ implies also that $C_o$ is self-adjoint.
         
         The conclusion of spectra follows from e.g. \cite[Section~4]{JDSP2}.
      \end{proof}
		
	\section{Proof of the main theorem} \label{sec:TheoremA}
		Given a generalized eigenvector $\{ u_n \}$ and a positive sequence $\{ \alpha_n \}$ we set
		\begin{align} \label{SN1}
			S_n = a_{n-1} \alpha_{n-1} u_{n-1}^2 + a_n \alpha_n u_n^2 - (\lambda - b_n) \alpha_{n-1} u_{n-1} u_n \quad (n \geq 1).
		\end{align}
		Using the identity $a_{n-1} u_{n-1} = (\lambda - b_n)u_n - a_n u_{n+1}$ we get an equivalent formula
		\begin{align} \label{SN2}
			S_n = \frac{\alpha_{n-1}}{a_{n-1}} a_n^2 u_{n+1}^2 + a_n \alpha_n u_n^2 - \frac{\alpha_{n-1}}{a_{n-1}} a_n (\lambda - b_n) u_{n+1} u_n \quad (n \geq 1).
		\end{align}
		
      The sequence $S_n$ for $\alpha_n = a_n$ was previously used in the study of Jacobi matrices, but only in the case of \emph{bounded} ones (see e.g. \cite{JD3}, \cite{DN}). In the case of unbounded operators a sequence similar to $S_n$ for $\alpha_n \equiv 1$ was also used in \cite{SLC}.
		
		The following proposition is an adaptation of \cite[Lemma 3.1]{SLC}.
		\begin{fakt} \label{faktAsymptotyka}
			Let $\{ u_n \}$ be a generalized eigenvector associated with $\lambda \in \mathbb{R}$ and
			\begin{align*}
				\widetilde{S}_n &= u_{n+1}^2 + u_n^2.
			\end{align*}
         Assume that $a_n \rightarrow \infty$, and
			\begin{align*}
				\lim_{n \rightarrow \infty} \frac{\alpha_{n-1}}{\alpha_n} \frac{a_n}{a_{n-1}} = 1, \quad \limsup_{n \rightarrow \infty} \frac{|b_n|}{a_n} < 2.
			\end{align*}
			Then there exist constants $c_1>0, c_2>0$ such that for sufficiently large $n$
			\[
				c_1 a_n \alpha_n \leq \frac{S_n}{\widetilde{S}_n} \leq c_2 a_n \alpha_n.
			\]
		\end{fakt}
		\begin{proof}
			Observe that from the representation~\eqref{SN2} we have that $S_n$ is a quadratic form with respect to variables $u_n$ and $u_{n+1}$. Let the minimal and the maximal value of $S_n$ under the condition $\widetilde{S}_n = 1$ be denoted by $w_n^{\text{min}}$ and $w_n^{\text{max}}$ respectively. Then
         \begin{align*}
            \frac{2 w_n^{\text{min}}}{a_n \alpha_n} &= 1 + \frac{\alpha_{n-1}}{\alpha_n} \frac{a_n}{a_{n-1}} - \sqrt{\left(1 - \frac{\alpha_{n-1}}{\alpha_n} \frac{a_n}{a_{n-1}} \right)^2 + \left(\frac{\alpha_{n-1}}{\alpha_n} \frac{a_n}{a_{n-1}} \frac{\lambda - b_n}{a_n} \right)^2},\\
            \frac{2 w_n^{\text{max}}}{a_n \alpha_n} &= 1 + \frac{\alpha_{n-1}}{\alpha_n} \frac{a_n}{a_{n-1}} + \sqrt{\left(1 - \frac{\alpha_{n-1}}{\alpha_n} \frac{a_n}{a_{n-1}} \right)^2 + \left(\frac{\alpha_{n-1}}{\alpha_n} \frac{a_n}{a_{n-1}} \frac{\lambda - b_n}{a_n} \right)^2}.
         \end{align*}
         Letting $n \rightarrow \infty$ we see that for large $n$ there is a positive upper and lower bound of the above expressions. What ends the proof.
		\end{proof}
		
		\begin{wn} \label{wniosekDodatniKomutator}
			Under the assumptions of Proposition~\ref{faktAsymptotyka}, together with
			\[
				\quad \sum_{n=0}^\infty \frac{1}{a_n \alpha_n} = \infty,
			\]
			if $\liminf S_n > 0$ then $u \notin \ell^2$.
		\end{wn}
		\begin{proof}
			Since $\liminf S_n > 0$ by Proposition~\ref{faktAsymptotyka} there exists a constant $c > 0$ such that for every $n$ sufficiently large we have
			\[
				\frac{c}{a_n \alpha_n} \leq \widetilde{S}_n
			\]
			what ends the proof.
		\end{proof}
		
		Now we are ready to prove Theorem~\ref{twSpektrumOgolne}.
		\begin{proof}[Proof of Theorem~\ref{twSpektrumOgolne}]
			By virtue of Corollary~\ref{wniosekDodatniKomutator} it is enough to show that $\liminf S_n > 0$ for every generalized eigenvector $\{ u_n \}$.
			
			By Proposition~\ref{faktAsymptotyka} there exists $N$ such that for every $n \geq N$ holds $S_n > 0$. Let us define $F_n = (S_{n+1} - S_n) / S_n$. Then $S_{n+1} / S_n = 1 + F_n$, thus
			\[
				\frac{S_n}{S_N} = \prod_{k=N}^{n-1} (1 + F_n).
			\]
			Hence
			\begin{equation} \label{sumowalnoscFN}
				\sum_{n=1}^\infty F_n^- < \infty.
			\end{equation}
         implies $\liminf S_n > 0$. Observe that by \eqref{SN1} and \eqref{SN2} we get
			\[
				S_{n+1} - S_n = \left( a_{n+1} \alpha_{n+1} - \frac{\alpha_{n-1}}{a_{n-1}} a_n^2 \right) u_{n+1}^2 + \left( \frac{\alpha_{n-1}}{a_{n-1}} a_n (\lambda - b_n) - \alpha_n (\lambda - b_{n+1}) \right) u_{n+1} u_n.
			\]
			Therefore
			\begin{multline*}
				F_n = \frac{S_{n+1} - S_n}{S_n} = \Bigg[ \left( a_{n+1} \alpha_{n+1} - \frac{\alpha_{n-1}}{a_{n-1}} a_n^2 \right) \frac{u_{n+1}^2}{\widetilde{S}_n} \\
				+ \left( \frac{\alpha_{n-1}}{a_{n-1}} a_n (\lambda - b_n) - \alpha_n(\lambda - b_{n+1}) \right) \frac{u_n u_{n+1}}{\widetilde{S}_n} \Bigg] \frac{\widetilde{S}_n}{S_n},
			\end{multline*}
         where $\widetilde{S}_n = u_n^2 + u_{n+1}^2$.
			By Proposition~\ref{faktAsymptotyka} and $|u_n u_{n+1}| / \widetilde{S}_n \leq 1$, there exists a constant $c > 0$ such that
			\[
				F_n^- \leq \frac{c}{a_n \alpha_n} \left( \left[ a_{n+1} \alpha_{n+1} - \frac{\alpha_{n-1}}{a_{n-1}} a_n^2 \right]^- + \left|\frac{\alpha_{n-1}}{a_{n-1}} a_n (\lambda - b_n) - \alpha_n (\lambda - b_{n+1}) \right| \right).
			\]
			Since
			\[
				\frac{1}{a_n \alpha_n} \left[ a_{n+1} \alpha_{n+1} - \frac{\alpha_{n-1}}{a_{n-1}} a_n^2 \right]^- = \left[ \frac{a_{n+1}}{a_n} \frac{\alpha_{n+1}}{\alpha_n} - \frac{a_n}{a_{n-1}} \frac{\alpha_{n-1}}{\alpha_n} \right]^-
			\]
			and 
			\begin{multline*}
				\frac{1}{a_n \alpha_n} \left|\frac{\alpha_{n-1}}{a_{n-1}} a_n (\lambda - b_n) - \alpha_n (\lambda - b_{n+1}) \right| = \left| \lambda \left( \frac{1}{a_{n-1}} \frac{\alpha_{n-1}}{\alpha_n} - \frac{1}{a_n} \right) + \left( \frac{b_{n+1}}{a_n} -\frac{b_n}{a_{n-1}} \frac{\alpha_{n-1}}{\alpha_n} \right)\right| \\
            \leq \frac{|\lambda|}{a_{n-1}} \left| \frac{\alpha_{n-1}}{\alpha_n} - \frac{a_{n-1}}{a_n} \right| + \left| \frac{b_{n+1}}{a_n} - \frac{b_n}{a_{n-1}} \frac{\alpha_{n-1}}{\alpha_n} \right|
			\end{multline*}
			we obtain~\eqref{sumowalnoscFN}.
		\end{proof}
		
		\begin{uwaga} \label{UwagaSubordynacja}
			If we replace the condition~(b) by
			\begin{flalign}
				\tag{b'} &\sum_{n=0}^\infty \left| \frac{a_{n+1}}{a_n} \frac{\alpha_{n+1}}{\alpha_n} - \frac{a_n}{a_{n-1}} \frac{\alpha_{n-1}}{\alpha_n} \right| < \infty,&
			\end{flalign}
			then $\limsup S_n < \infty$ and consequently $c_1/(a_n \alpha_n) \leq \widetilde{S}_n \leq c_2/(a_n \alpha_n)$ for $c_1>0, c_2>0$. Hence by using subordination method we can show that the spectrum of the matrix~$C$ is purely absolutely continuous (see e.g. \cite{SLC}, \cite{JN1}).
		\end{uwaga}
	
	\section{Special cases} \label{sec:specialCases}
      In this section we are going to show a few choices of the sequence $\{ \alpha_n \}$ from Theorem~\ref{twSpektrumOgolne}. In this way we show flexibility of our approach.
      
      The following theorem was proven in \cite[Theorem~1.6]{JM1} and is a generalization of \cite[Theorem~1.10]{SLC}. In the proof the authors analyse transfer matrices. Therefore our argument gives an alternative proof.
		\begin{tw}[Janas, Moszyński \cite{JM1}] \label{twClark}
			Assume that 
         \begin{flalign}
            \tag{a} &\lim_{n \rightarrow \infty} a_n = \infty, &\\
            \tag{b} &\sum_{n=0}^\infty \frac{1}{a_n} = \infty, \\
            \tag{c} &\text{the sequences } \left\{ \frac{a_{n-1}}{a_n} \right\}, \left\{ \frac{1}{a_n} \right\} \text{ and } \left\{ \frac{b_{n}}{a_n} \right\} \text{ are of bounded variation}, \\
            \tag{d} &\lim_{n \rightarrow \infty} \frac{|b_n|}{a_n} < 2.
         \end{flalign}
			Then $\sigma(C) = \mathbb{R}$ and the matrix~$C$ has purely absolutely continuous spectrum.
		\end{tw}
		\begin{proof}
         Let $\alpha_n \equiv 1$. By virtue of Remark~\ref{UwagaSubordynacja} we need to check the assumptions (b'), (d) and (f) of Theorem~\ref{twSpektrumOgolne}.
         
         Since the sequence $\{ a_{n-1}/a_n \}$ is of bounded variation it is convergent to a number $a$. From the condition~(b) we have $a \geq 1$, whereas the condition~(a) gives $a \leq 1$. Thus the sequence $\{ a_{n+1}/a_n \}$ is of bounded variation as well. This proves the conditions (b') and (f) of Theorem~\ref{twSpektrumOgolne}.
         
         The sequence $\{ b_{n+1}/a_n \}$ is of bounded variation because $\frac{b_{n+1}}{a_n} = \frac{b_{n+1}}{a_{n+1}} \cdot \frac{a_{n+1}}{a_n}$. The proof is complete.
		\end{proof}

      The next theorem imposes very simple conditions on Jacobi matrices. In Section~\ref{sec:TheoremAAppl} we show its applications, furthermore in Section~\ref{sec:Examples} we discuss sharpness of the assumptions.
      
		\begin{tw} \label{twSpektrumAKwadrat}
			Assume
         \begin{flalign}
            \tag{a} &\lim_{n \rightarrow \infty} a_n = \infty, &\\
            \tag{b} &\sum_{n=0}^\infty \frac{1}{a_n^2} = \infty, \\
            \tag{c} &\sum_{n=0}^\infty \left[ \left( \frac{a_{n+1}}{a_n} \right)^2 - 1 \right]^- < \infty, \\
            \tag{d} &\limsup_{n \rightarrow \infty} \frac{|b_n|}{a_n} < 2, \\
            \tag{e} &\sum_{n=0}^\infty \frac{|b_{n+1} - b_n|}{a_n} < \infty.
         \end{flalign}
         Then the Jacobi matrix $C$ is self-adjoint and satisfies $\sigma_p(C) = \emptyset$ and $\sigma(C) = \mathbb{R}$.
		\end{tw}
		\begin{proof}
			Apply Theorem~\ref{twSpektrumOgolne} with $\alpha_n = a_n$.
		\end{proof}
		
		Special cases of the following theorem were examined in \cite{SP} and \cite{JM1} using commutator methods.
		\begin{tw} \label{twJanasMoszynskiPedersen}
			Let $\log^{(i)}$ be defined by $\log^{(0)}(x) = x, \log^{(i+1)}(x) = \log(\log^{(i)}(x))$. Let $g_j(n) = \prod_{i=1}^j \log^{(i)}(n)$. Assume that for positive numbers $K, N$ and for a summable nonnegative sequence~$c_n$
         \begin{flalign}
            \tag{a} &\lim_{n \rightarrow \infty} a_n = \infty, &\\
            \tag{b} &1 - c_n \leq \frac{a_n}{a_{n-1}} \leq 1 + \frac{1}{n} + \sum_{j=1}^K \frac{1}{n g_j(n)} + c_n \text{ for } n > N, \\
            \tag{c} &\text{the sequence } \{ b_n \} \text{ is bounded and } \sum_{n=0}^\infty \frac{|b_{n+1} - b_n|}{a_n} < \infty, \\
            \tag{d} &\sum_{n=1}^\infty \frac{1}{n a_n} < \infty.
         \end{flalign}
			Then $\sigma_p(C) = \emptyset$ and $\sigma(C) = \mathbb{R}$.
		\end{tw}
		\begin{proof}
         We can assume that $\log^{(K)}(N) > 0$. Set
         \[
            \alpha_n = 
            \begin{cases}
               1 & \text{ for } n < N, \\
               \frac{n g_K(n)}{a_n} & \text{ otherwise.}
            \end{cases}
         \] 
         To get the conclusion we need to check the assumptions (b), (d) and (c) of Theorem~\ref{twSpektrumOgolne}.
			
			To show Theorem~\ref{twSpektrumOgolne}(b) let us observe that the assumption~(b) of the present theorem gives 
         \[
            \left( \frac{a_n}{a_{n-1}} \right)^2 \leq 1 + \frac{2}{n} + \sum_{j=1}^K \frac{2}{n g_j(n)} + c_n'
         \] 
         for a summable sequence $c_n'$. Therefore
			\begin{multline*}
				\frac{a_{n+1}}{a_n} \frac{\alpha_{n+1}}{\alpha_n} - \frac{a_n}{a_{n-1}} \frac{\alpha_{n-1}}{\alpha_n} = \frac{n+1}{n} \frac{g_K(n+1)}{g_K(n)} - \left( \frac{a_n}{a_{n-1}} \right)^2 \frac{n-1}{n} \frac{g_K(n-1)}{g_K(n)} \\
				\geq \frac{n+1}{n} \frac{g_K(n+1)}{g_K(n)} - \frac{n-1}{n} \left( 1 + \frac{2}{n} + \sum_{j=1}^K \frac{2}{n g_j(n)} + c_n' \right) \frac{g_K(n-1)}{g_K(n)} \\
            \geq \frac{n+1}{n} \frac{g_K(n+1)}{g_K(n)} - \left( \frac{n+1}{n} + \sum_{j=1}^K \frac{2}{n g_j(n)} + c_n' \right) \frac{g_K(n-1)}{g_K(n)}.
			\end{multline*}
			Since the functions $g_j$ are increasing, we have
			\begin{equation} \label{rownJM1}
				\geq \frac{n-1}{n} \left( \frac{g_K(n+1) - g_K(n-1)}{g_K(n)} \right) - \frac{g_K(n-1)}{n g_K(n)} \sum_{j=1}^K \frac{2}{g_j(n-1)} - c_n'.
			\end{equation}
         Next, observe that
         \[
            g'_K(x) = g_K(x) \sum_{j=1}^K \frac{(\log^{(j)})'(x)}{\log^{(j)}(x)}.
         \]
         Therefore 
         \[
            g_K'(x) = g_K(x) \sum_{j=1}^K \frac{1}{x g_j(x)}.
         \]
         Hence Taylor's formula applied to $g_K$ at the point $n-1$ gives
         \[
            (n-1)[g_K(n+1) - g_K(n-1)] = g_K(n-1) \sum_{j=1}^K \frac{2}{g_j(n-1)} + 2 (n-1)g_K''(\xi)
         \]
         for $\xi \in (n-1, n+1)$. Direct computation shows $|g_K''(x)| \leq c/x^{3/2}$ for $x$ sufficiently large and a constant $c > 0$. Therefore the right-hand side of \eqref{rownJM1} is summable.
			
			Next, since
         \[
            \frac{b_{n+1}}{a_n} - \frac{b_n}{a_{n-1}} \frac{\alpha_{n-1}}{\alpha_n} = \frac{b_{n+1} - b_n}{a_n} + \frac{b_n}{a_{n-1}} \left( \frac{a_{n-1}}{a_n} - \frac{\alpha_{n-1}}{\alpha_n} \right)
         \]
         the condition Theorem~\ref{twSpektrumOgolne}(d) reduces to showing Theorem~\ref{twSpektrumOgolne}(c):
			\begin{equation} \label{warunekC}
				\sum_{n=0}^\infty \frac{1}{a_{n-1}} \left| \frac{a_{n-1}}{a_n} - \frac{n-1}{n} \frac{g_K(n-1)}{g_K(n)} \frac{a_n}{a_{n-1}} \right| < \infty.
			\end{equation}
			For constants $K'$ and $c > 0$ we have
			\begin{multline*}
				\frac{a_{n-1}}{a_n} - \frac{n-1}{n} \frac{g_K(n-1)}{g_K(n)} \frac{a_n}{a_{n-1}} \geq \frac{1}{1 + \frac{K'}{n} + c_n} - \left( 1 - \frac{1}{n} \right) \left( 1 + \frac{K'}{n} + c_n \right) \geq -\frac{c}{n} - c_n'
			\end{multline*}
         for a summable sequence $c_n'$. On the other hand 
         \begin{multline*}
				\frac{a_{n-1}}{a_n} - \frac{n-1}{n} \frac{g_K(n-1)}{g_K(n)} \frac{a_n}{a_{n-1}} \leq \frac{1}{1-c_n} - \left( 1 - \frac{1}{n} \right) \frac{g_K(n-1)}{g_K(n)} (1-c_n) \\
            = 1 - \frac{g_K(n-1)}{g_K(n)} + c_n' = \frac{g_K(n) - g_K(n-1)}{g_K(n)} + c_n'
			\end{multline*}
         for a summable sequence $c_n'$. Hence as previously Taylor's formula applied to $g_K$ at the point $n-1$ gives
         \[
            \frac{a_{n-1}}{a_n} - \frac{n-1}{n} \frac{g_K(n-1)}{g_K(n)} \frac{a_n}{a_{n-1}} \leq \frac{c}{n} + c''_n
         \]
         for a constant $c > 0$ and summable sequence $c''_n$. Finally, condition~(d) leads to \eqref{warunekC}.
		\end{proof}
		\begin{uwaga}
         When we compare Theorem~\ref{twSpektrumAKwadrat} with Theorem~\ref{twJanasMoszynskiPedersen}, we see that Theorem~\ref{twJanasMoszynskiPedersen} is interesting only in the case when $\sum_{n=0}^\infty 1/a_n^2 < \infty$. In this case the condition Theorem~\ref{twJanasMoszynskiPedersen}(d) is satified.
		\end{uwaga}
		
		The sequence similar to $\alpha_n = n a_n^{-1}$ was used in the proof of \cite[Theorem 4.1]{SP} and \cite[Theorem 2.1]{JM1}. There was shown that under the stronger assumptions (which in particular imply $c_n \equiv 0$, $b_n \equiv 0$ and $K=0$) the measure $\mu$ is absolutely continuous. Whether $\sigma(C) = \mathbb{R}$ was not investigated.
	
      \begin{przyklad}
         Let $K > 0$. Fix $M$ such that $\log^{(K)}(M) > 0$. Then for the sequences $a_n = (n+M) g_K(n+M)$ and $b_n \equiv 0$ the assumptions of Theorem~\ref{twJanasMoszynskiPedersen} are satisfied.
      \end{przyklad}
   
	\section{Applications of Theorem~\ref{twSpektrumAKwadrat}} \label{sec:TheoremAAppl}
      \subsection{Birth and death processes}
         Given sequences $\{\lambda_n\}_{n=0}^\infty$ and $\{\mu_n\}_{n=0}^\infty$ such that $\lambda_n > 0, \mu_{n+1} > 0 \  (n \geq 0)$ and $\mu_0 \geq 0$ we set
         \begin{equation} \label{postacUS}
            Q =
               \left( 
                  \begin{array}{cccccc}
                  -(\lambda_0 + \mu_0) & \lambda_0 & 0   & 0   &\ldots \\
                  \mu_1 & -(\lambda_1 + \mu_1) & \lambda_1 & 0 & \ldots \\
                  0   & \mu_2 & -(\lambda_2 + \mu_2) & \lambda_2 & \ldots \\
                  0   & 0   & \mu_3 & -(\lambda_3 + \mu_3) & \ldots \\
                  \vdots & \vdots & \vdots & \vdots & \ddots
                  \end{array} 
               \right).
         \end{equation}
         Let us define
         \[
            \ell^2(\pi) = \{ x \in \mathbb{C}^\mathbb{N} \colon \sum_{n=0}^\infty \pi_n |x_n|^2 < \infty \}, \quad \langle x, y \rangle_{\ell^2(\pi)} = \sum_{n=0}^\infty \pi_n x_n \overline{y_n}
         \]
         where
         \[
            \pi_0 = 1, \quad \pi_n = \frac{\lambda_0 \lambda_1 \ldots \lambda_{n-1}}{\mu_1 \mu_2 \ldots \mu_n}.
         \]
         
         The operator~$Q$ is well-defined on the domain $\Dom(Q) = \{ x \in \ell^2(\pi) \colon Q x \in \ell^2(\pi)\}$. Notice that any sequence with finite support belongs to $\Dom(Q)$. If the operator~$Q$ is self-adjoint it is of a probabilistic interest to examine the spectrum~$\sigma(Q)$ of the operator~$Q$ (see e.g. \cite{SKJMG}).
         
         \begin{tw} \label{twBDProcess}
            Let $a = (\mu_1, \lambda_1, \mu_2, \lambda_2, \mu_3, \lambda_3, \ldots)$. Assume
            \begin{flalign}
               \tag{a} &\lim_{n \rightarrow \infty} a_n = \infty, &\\
               \tag{b} &\sum_{n=0}^\infty \frac{1}{a_n} = \infty, \\
               \tag{c} &\sum_{n=0}^\infty \left[ \frac{a_{n+1}}{a_n} - 1 \right]^- < \infty.
            \end{flalign}
            Then the matrix~$Q$ is self-adjoint and satisfies $\sigma_p(Q) = \emptyset$ and $\sigma(Q) = (-\infty, 0]$.
         \end{tw}
         \begin{proof}
            Let $P$ be a diagonal matrix with entries $\sqrt{\pi_n}$ on the main diagonal. Then we have $\bar{C} = P Q P^{-1}$, where $\bar{C}$ is the Jacobi matrix associated with sequences $\bar{a}_n = \sqrt{\lambda_n \mu_{n+1}}$ and $\bar{b}_n = -(\lambda_n + \mu_n)$ (see \cite[Section~2]{MK}). Since the matrix~$P \colon \ell^2(\pi) \rightarrow \ell^2$ is an isometry (hence $P$ and $P^{-1}$ are bounded) it is enough to consider only the spectrum of $\bar{C}$. By virtue of Proposition~\ref{odbicieSpektrum} it is sufficient to consider the spectrum of the matrix $\widehat{C}$, corresponding with the sequences $\{a_n\}$ and $\{-b_n\}$.
            
            Let us consider the case $\mu_0 = 0$. Let $\widetilde{b}_n \equiv 0$ and
            \[
               \widetilde{a} = (\sqrt{\lambda_0}, \sqrt{\mu_1}, \sqrt{\lambda_1}, \sqrt{\mu_2}, \sqrt{\lambda_2}, \ldots).
            \] 
            Observe that by Proposition~\ref{restrykcjeC} we have $\widetilde{C}_e = \widehat{C}$. Hence, by Theorem~\ref{twSpektrumAKwadrat} the conclusion follows.
            
            Next, suppose that $\mu_0 > 0$. Let $\widetilde{b}_n \equiv 0$ and
            \[
               \widetilde{a} = (\sqrt{\mu_0}, \sqrt{\lambda_0}, \sqrt{\mu_1}, \sqrt{\lambda_1}, \sqrt{\mu_2}, \sqrt{\lambda_2}, \ldots).
            \] 
            Applying Proposition~\ref{restrykcjeC} to $\widetilde{C}_o = \widehat{C}$ by Theorem~\ref{twSpektrumAKwadrat} we finish the proof.
         \end{proof}
         
         In \cite{BRGV} the following conjecture about spectral properties of operators of the form \eqref{postacUS} was stated.
         
         \begin{hipoteza}[Roehner and Valent \cite{BRGV}] \label{hipotezaValent}
            Assume that
            \[
               \lim_{n \rightarrow \infty} \mu_n/\lambda_n = 1, \quad \lim_{n \rightarrow \infty} \lambda_n/n^\alpha = a
            \]
            for constants $a > 0$ and $0 < \alpha \leq 2$. Then $\sigma_p(Q) = \emptyset$.
         \end{hipoteza}
         
         In \cite{Chihara3} it was shown that without additional assumptions the conjecture is false. In Theorem~\ref{twBDProcess} we provide sufficient conditions when Conjecture~\ref{hipotezaValent} holds.
         
         It is worthwhile to compare Theorem~\ref{twBDProcess} with results obtained in \cite{MK}. Let
         \[
            \lim_{n \rightarrow \infty} \mu_n / \lambda_n = q \quad (0 < q < \infty).
         \]
         Then in \cite{MK} was concluded that under additional assumptions (which in particular imply $\lambda_{k+1}/\lambda_k \rightarrow 1, \ \mu_{k+1} / \mu_k \rightarrow 1$, $\lambda_k \rightarrow \infty$ and $\alpha < 1$) the matrix~Q satisfies $\sigma_{ess}(Q) = \emptyset$. However, there is a problem in the proof of Lemma~1(iii) on the page~69. The author states that $\lVert F D^{-1} \rVert_{\ell^2} < 1$ if for a certain $\zeta > 0$
         \[
            \frac{\sqrt{\lambda_k \mu_{k+1}}}{\lambda_k + \mu_k + \zeta} < \frac{1}{2}, \quad \frac{\sqrt{\lambda_k \mu_{k+1}}}{\lambda_{k+1} + \mu_{k+1} + \zeta} < \frac{1}{2}.
         \] 
         In fact what we need is
         \[
            \frac{\sqrt{\lambda_k \mu_{k+1}}}{\lambda_k + \mu_k} < \frac{1}{2} - \epsilon, \quad \frac{\sqrt{\lambda_k \mu_{k+1}}}{\lambda_{k+1} + \mu_{k+1}} < \frac{1}{2} - \epsilon
         \] 
         for certain $\epsilon > 0$, which, under the assumption $q = 1$ is impossible because the left-hand sides converge to $1/2$. In fact Theorem~\ref{twBDProcess} implies the \emph{opposite} conclusion to results from \cite{MK}.
         
         Note that
         \[
            \lim_{n \rightarrow \infty} \frac{\lambda_n \mu_{n+1}}{(\lambda_n + \mu_n) (\lambda_{n+1} + \mu_{n+1})} = \frac{q}{(1+q)^2}
         \]
         what under the assumption $q \neq 1$ is strictly less than $1/4$. Therefore \cite[Theorem~1]{Chihara3} (for a functional analytic proof see \cite[Theorem~2.6]{Szwarc}) combined with Proposition~\ref{odbicieSpektrum} implies that if the matrix~$Q$ is self-adjoint and $\lambda_k \rightarrow \infty$ then $\sigma_{ess}(Q) = \emptyset$.
         
      \subsection{Chihara's conjecture}
         In \cite{Chihara4} (see also \cite[IV-Theorem~4.2]{Chihara5}) the following result was proven.
         \begin{tw}[Chihara \cite{Chihara4}] \label{twChihara14}
            Assume that a Jacobi matrix~$C$ is self-adjoint, $b_n \rightarrow \infty$, the smallest point $\rho$ of $\sigma_{ess}(C)$ is finite and
            \[
               \lim_{n \rightarrow \infty} \frac{a_n^2}{b_n b_{n+1}} = \frac{1}{4}.
            \]
            Then the set $\{ x \colon p_n(x) = 0, n \in \mathbb{N} \}$ of the zeros of orthogonal polynomials $\{ p_n \}$ is dense in $[\rho, \infty)$.
         \end{tw}
         It suggests the following conjecture stated in \cite{Chihara1} and \cite{Chihara2}.
         
         \begin{hipoteza}[Chihara \cite{Chihara1}, \cite{Chihara2}] \label{hipChihara}
            Let the assumptions of Theorem~\ref{twChihara14} be satisfied. Then $\sigma_{ess}(C) = [\rho, \infty)$.
         \end{hipoteza}
         
         The following theorem gives sufficient (and easy to verify) additional conditions for Conjecture~\ref{hipChihara} to hold. In fact every Jacobi matrix with $b_n \equiv 0$ and $a_{n+1}/a_n \rightarrow 1$ from this article provides an example (via Proposition~\ref{restrykcjeC}) when Conjecture~\ref{hipChihara} holds.
         
         \begin{tw} \label{twChihara}
            Assume
            \begin{flalign}
               \tag{a} &\lim_{n \rightarrow \infty} a_n = \infty, &\\
               \tag{b} &\sum_{n=0}^\infty \frac{1}{a_n} = \infty, \\
               \tag{c} &\sum_{n=0}^\infty \left[ \frac{a_{n+1}}{a_n} - 1 \right]^- < \infty, \\
               \tag{d} &\lim_{n \rightarrow \infty} [a_{n-1} - b_n + a_n] = M.
            \end{flalign}
            Then the Jacobi matrix $C$ satisfies $\sigma_{ess}(C) = [-M, \infty)$. Moreover, if $a_{n+1}/a_n \rightarrow 1$ then
            \begin{equation} \label{granica14}
               \lim_{n \rightarrow \infty} \frac{a_n^2}{b_n b_{n+1}} = \frac{1}{4}.
            \end{equation}
         \end{tw}
         \begin{proof}
            We show~\eqref{granica14} by a direct computation. Without loss of generality we may assume that $M = 0$. Let $-r_n = a_{n-1} - b_n + a_n$. Then $a_{n-1} - (b_n-r_n) + a_n = 0$. Let $\widetilde{C}$ be the Jacobi matrix for sequences $\widetilde{a}_n = a_n, \ \widetilde{b}_n = b_n - r_n$. The matrix $R = C - \widetilde{C}$ defines a compact self-adjoint operator on $\ell^2$ (because $r_n \rightarrow 0$). Hence, by the Weyl perturbation theorem (see \cite{Weyl}), $\sigma_{ess}(C) = \sigma_{ess}(\widetilde{C})$. Theorem~\ref{twBDProcess} implies that $\sigma_{ess}(\widetilde{C}) = (-\infty, 0]$. Finally, Proposition~\ref{odbicieSpektrum} applied to the matrix $\widetilde{C}$ finishes the proof.
         \end{proof}
         
	\section{Examples} \label{sec:Examples}
		\begin{przyklad}
         Let $b_n \equiv 0, \ \epsilon > 0, \ a_0 = \epsilon$ and $a_{2k - 1} = a_{2k} = \widetilde{a}_k \ (k \geq 1)$ for a sequence $\widetilde{a}_k, \  \widetilde{a}_k \rightarrow \infty$. Then the matrix~$C$ is always self-adjoint. Moreover, $0$ is its eigenvalue if and only if 
         \[
            \sum_{k=0}^\infty \left( \frac{a_0 a_2 \ldots a_{2k}}{a_1 a_3 \ldots a_{2k+1}} \right)^2 = \epsilon^2 \sum_{k=1}^\infty \frac{1}{\widetilde{a}_k^2} < \infty,
         \]
         (see e.g. \cite[Theorem 3.2]{JDSP2}). Therefore the condition Theorem~\ref{twSpektrumAKwadrat}(b) could not be weakened even for the class of monotonic sequences $a_n$. 
         
         In \cite{MM} it was shown that for $\widetilde{a}_k = k^\alpha, \ (\alpha \in (0,1))$ the spectrum $\sigma(C) = \mathbb{R}$. In case $\alpha \leq 1/2$ the measure~$\mu(\cdot) = \langle E(\cdot) \delta_0, \delta_0 \rangle$ is absolutely continuous, whereas for $\alpha > 1/2$ the measure $\mu$ is absolutely continuous on the set $\mathbb{R} \backslash \{ 0 \}$.
      \end{przyklad}
      
      \begin{przyklad}
         Let $b_n \equiv 0$ and $a_n = n^\alpha + c_n \ (0 < \alpha \leq 2/3)$ where $c_{2n} = 1$ and $c_{2n+1} = 0$. Then (see \cite{DJMP}) $\sigma(C) = \mathbb{R} \backslash (-1,1)$ and the measure~$\mu$ is absolutely continuous on $\mathbb{R} \backslash [-1,1]$. It shows that the condition Theorem~\ref{twSpektrumAKwadrat}(c) could not be replaced by $[(a_{n+1}/a_n)^2 - 1]^- \rightarrow 0$.
      \end{przyklad}
		
		\begin{przyklad}
			Let $a_0 = 1$ and for $k! \leq n < (k+1)!$ we define $a_n = \sqrt{k!}$. For $n > 0$ we have
			\[
				\frac{a_{n+1}}{a_n} = 
				\begin{cases}
					\sqrt{k} & \text{if $n+1=k!$} \\
					1 & \text{otherwise.}
				\end{cases}
			\]
			Define $b_n \equiv 0$. We have $a_n \leq \sqrt{n+1}$. Therefore $\sum_{n=0}^\infty 1/a_n^2 = \infty$. Observe that the assumptions of Theorem~\ref{twSpektrumAKwadrat} are satisfied. Moreover, $a_{n+1}/a_n \nrightarrow 1$ and \cite[Theorem~3.1]{JN1} nor \cite[Lemma~2.6]{JDSP1} cannot be applied.
		\end{przyklad}

   \bibliographystyle{plain}
   \bibliography{Jacobi}
\end{document}